\documentclass[review]{elsarticle}

\usepackage{lineno,hyperref}
\usepackage{amssymb,amsmath}
\usepackage{amsthm}

\newtheorem{theorem}{Theorem}

\newcommand{\be}{\begin{equation}}
\newcommand{\ee}{\end{equation}}
\def\bee#1\eee{\begin{align}#1\end{align}}
\newcommand{\bse}{\begin{subequations}}
\newcommand{\ese}{\end{subequations}}
\newcommand{\nnb}{\nonumber}

\journal{Journal of Combinatorial Theory, Series A}

\begin{document}

\begin{frontmatter}

\title{Convex Set of Doubly Substochastic Matrices}

%

\author[mymainaddress]{Lei Deng\corref{mycorrespondingauthor}}
\cortext[mycorrespondingauthor]{Corresponding author}
\ead{denglei@dgut.edu.cn}

\author[mysecondaryaddress]{Qiulin Lin}
\ead{lq016@ie.cuhk.edu.hk}

\address[mymainaddress]{School of Electrical Engineering \& Intelligentization, Dongguan University of Technology,
No.~1, Daxue Road, Songshan Lake, Dongguan 523808, China}
\address[mysecondaryaddress]{Department of Information Engineering, The Chinese University of Hong Kong, Shatin, New Territories, Hong Kong, China}

\begin{abstract}
Denote $\mathcal{A}$ as the set of all doubly substochastic $m \times n$ matrices and let $k$ be a positive integer.
Let $\mathcal{A}_k$ be the set of all $1/k$-bounded doubly substochastic $m \times n$ matrices, i.e., $\mathcal{A}_k \triangleq \{E \in \mathcal{A}: e_{i,j} \in [0, 1/k], \forall i=1,2,\cdots,m, j = 1,2,\cdots, n\}$.
Denote $\mathcal{B}_k$ as the set of all matrices in $\mathcal{A}_k$ whose entries are either $0$ or $1/k$.
We prove that $\mathcal{A}_k$ is the convex hull of all matrices in $\mathcal{B}_k$.
\end{abstract}

\begin{keyword}
Doubly substochastic matrix \sep Convex hull \sep Extreme points
\end{keyword}

\end{frontmatter}


\section{Introduction}
An $n \times n$ matrix $E=(e_{i,j})$ is \emph{doubly stochastic} if it satisfies the following conditions:
\be
\left\{
  \begin{array}{ll}
    e_{i,j} \ge 0, & \hbox{$\forall i,j=1,2,\cdots, n$;} \\
    \sum_{j=1}^{n} e_{i,j} = 1, & \hbox{$\forall i = 1,2,\cdots,n$;} \\
    \sum_{i=1}^{n} e_{i,j} = 1, & \hbox{$\forall j = 1,2,\cdots,n$.}
  \end{array}
\right.
\nnb
\ee
Birkhoff in \cite{birkhoff} shows that the set of all $n \times n$ doubly stochastic matrix is the convex hull
of all $n \times n$ permutation matrices. Recall that a permutation matrix is
a square $(0,1)$ matrix of which any line (row or column) has exactly one 1.

An $m \times n$ matrix  $E=(e_{i,j})$ is \emph{doubly substochastic} if it satisfies the following conditions:
\be
\left\{
  \begin{array}{ll}
    e_{i,j} \ge 0, & \hbox{$\forall i = 1,2,\cdots,m , j = 1,2,\cdots,n$;} \\
    \sum_{j=1}^{n} e_{i,j} \le 1, & \hbox{$\forall i = 1,2,\cdots,n$;} \\
    \sum_{i=1}^{n} e_{i,j} \le 1, & \hbox{$\forall j = 1,2,\cdots,m$.}
  \end{array}
\right.
\nnb
\ee
Mirsky in \cite{mirsky1959convex} shows that the set of all $n \times n$ doubly substochastic matrix is the convex hull
of all $n \times n$ subpermutation matrices. Recall that a subpermutation
matrix is a $(0,1)$ matrix of which any line (row or column) has at most one 1.
It is straightforward to extend Mirsky's result to general $m \times n$ matrices
because we can always augment a rectangular doubly substochastic matrix to a square doubly substochastic matrix by adding some zero lines.

Each entry in doubly stochastic/substochastic matrices  can be any real number in the interval $[0,1]$.
However, if we impose an upper bound $1/k$ where $k$ is a positive integer $k$ for all entries of  doubly stochastic/substochastic matrices,
the convex-hull characterization  becomes different.
Watkins and Merris in \cite{watkins1974convex} prove that the set of all ${1}/{k}$-bounded doubly stochastic matrices
is the convex hull of all doubly stochastic matrices whose entries are either 0 or $1/k$.
In this work, we obtain the counterpart result of \cite{watkins1974convex} for doubly substochastic matrices.
Specifically, we prove that the set of all ${1}/{k}$-bounded doubly substochastic matrices
is the convex hull of all doubly substochastic matrices whose entries are either 0 or $1/k$.

\section{Our Result}
Denote $\mathcal{A}$ as the set of all doubly substochastic $m \times n$ matrices and let $k$ be a positive integer.
Define the set of all $1/k$-bounded doubly substochastic $m \times n$ matrices as $\mathcal{A}_k$, i.e.,
\[
\mathcal{A}_k \triangleq \{E \in \mathcal{A}: e_{i,j} \in [0, 1/k], \forall i=1,2,\cdots, m, j=1,2, \cdots, n\}.
\]
Denote $\mathcal{B}_k$ as the set of all matrices in $\mathcal{A}_k$ whose entries are either $0$ or $1/k$.
We now present our result.

\begin{theorem} \label{thm:main-result}
$\mathcal{A}_k$ is the convex hull of all matrices in $\mathcal{B}_k$.
\end{theorem}

\begin{proof}
It is straightforward to show that $\mathcal{A}_k$ is a non-empty compact convex set in $\mathbb{R}^{m\times n}$ and
any matrix in $\mathcal{B}_k$ is an extreme point of set $\mathcal{A}_k$.
Since any non-empty compact convex set in $\mathbb{R}^{n}$ is the convex hull of its extreme points \cite[Theorem 2.23]{giaquinta2011mathematical},
to finish the proof, we thus only need to show that any matrix $E \in \mathcal{A}_k \setminus \mathcal{B}_k$ is
not an extreme point of set $\mathcal{A}_k$.

An entry being neither 0 nor $1/k$ is called a \emph{middle} entry, and a row (resp. column) containing at least one middle entry is called a
\emph{middle} row (resp. column). A middle line is either a middle row or a middle column.
For any matrix $E \in \mathcal{A}_k \setminus \mathcal{B}_k$, it has at least one middle entry.  We consider two cases.

\emph{Case I: Any middle line has at least two middle entries.}
We aim at finding a chain of middle entries of matrix $E$.
We begin with any middle entry, denoted as $(r_1, c_1)$. Then row $r_1$ is a middle row which has at least two middle entries.
Hence, we can find a new middle entry, denoted as $(r_1, c_2)$.
Similarly, column $c_2$ is a middle column which has at least two middle entries.
Thus we can find another new middle entry $(r_2, c_2)$. Then we can find another new middle entry $(r_2, c_3)$ in middle row $r_2$.
If $c_3 = c_1$, we find the (latest-visited) old middle entry $(r_3,c_3)=(r_1,c_1)$ and terminate;
otherwise, we can find another new middle entry $(r_3, c_3)$. The process continues by alternating a middle row and
a middle column. In each new middle line, if there are old middle entries, we find the latest-visited old middle entry and terminate; otherwise,
we find a new middle entry and continue.
Since $E$ only has  a finite number
of middle entries, the process must terminate by finding an old middle entry, terminating either (i) at a middle entry $(r_{t'}, c_{t'})$
which coincides with an old middle entry $(r_t, c_t)$ or (ii) at a middle entry $(r_{t'}, c_{t'+1})$ which coincides with an old middle entry $(r_t, c_{t+1})$ for some positive integers $t$ and $t'$ such that $t < t'-1$.
Let us consider terminating-condition (i). The proof for terminating-condition (ii) is similar. In terminating-condition (i), we can find a loop of middle entries:
\[
(r_t, c_t) \to (r_t, c_{t+1}) \to (r_{t+1}, c_{t+1}) \to \cdots \to (r_{t'-1}, c_{t'}) \to  (r_{t'}, c_{t'})=(r_t, c_t),
\]
which has in total $2(t'-t)$ different middle entries.
Now we construct the following two matrices $E'$ and $E''$, both of which have the same entries of matrix $E$ except
\be
\left\{
  \begin{array}{ll}
    E'_{r_i, c_i} = E_{r_i, c_i} + \epsilon, E''_{r_i, c_i} = E_{r_i, c_i} - \epsilon, & \hbox{$ \forall i = t, t+1, \cdots, t'-1$;} \\
    E'_{r_i, c_{i+1}} = E_{r_i, c_{i+1}} - \epsilon, E''_{r_i, c_{i+1}} = E_{r_i, c_{i+1}} + \epsilon, & \hbox{$ \forall i = t, t+1, \cdots, t'-1$.}
  \end{array}
\right.
\nnb
\end{equation}

Under such construction, any row/column sum of $E'$ and $E''$ is equal to that of $E$.
We can always find a sufficiently small positive real number $\epsilon$ such that
all entries in $E'$ and $E''$ are in the interval $[0,1/k]$, ensuring $E' \in \mathcal{A}_k$ and $E'' \in \mathcal{A}_k$.
Since we have $E = (E'+E'')/2$ and $E' \neq E''$, matrix $E$ is not an extreme point of set $\mathcal{A}_k$.

\emph{Case II: There exists at least one middle line containing only one middle entry.}
We again aim at finding a chain of middle entries.
We first find a middle line (a middle row or a middle column) which has only one middle entry.
Let us assume that this is a middle column, say column $c_1$. The proof for a middle row is similar.
Denote the only middle entry in this middle column as $(r_1, c_1)$.
Now in the row $r_1$ we try to find another new middle entry. If
we cannot find such a middle entry, we terminate; otherwise, we denote the new middle entry as $(r_1, c_2)$ and we continue.
Similar to \emph{Case I}, the process continues by alternating a middle row and a middle column.
In each new middle line, if there are old middle entries, we find the latest-visited old middle entry and terminate;
otherwise, we try to find a new middle entry: we continue if indeed we can find a new middle entry but terminate if we cannot find any new middle entry.
Since $E$ only has a finite number of middle entries, the process
must terminate either (i)  when we find an old middle entry or (ii) when we cannot find any new middle entry. If it terminates at condition (i), we can again find a loop of middle entries
and use the proof of \emph{Case I} to show that $E$ is not an extreme points. If it terminates at condition (ii), it either terminates
at a middle entry $(r_t, c_t)$ or a middle entry $(r_t, c_{t+1})$ for some $t \ge 1$. We prove for the case $(r_t, c_t)$. The other case $(r_t, c_{t+1})$ has a similar proof.
If the process terminates at entry $(r_t,c_t)$, we cannot find an (old or new) middle entry in row $r_t$, i.e., $r_t$ is a middle row with only one middle entry $(r_t,c_t)$.
We now have a chain of $(2t-1)$ different middle entries,
\[
(r_1, c_1) \to (r_1, c_2) \to (r_2, c_2) \to \cdots \to (r_{t-1}, c_{t}) \to  (r_{t}, c_{t}).
\]
Now we construct the following two matrices $E'$ and $E''$, both of which have the same entries of matrix $E$ except
\be
\left\{
  \begin{array}{ll}
    E'_{r_i, c_i} = E_{r_i, c_i} + \epsilon, E''_{r_i, c_i} = E_{r_i, c_i} - \epsilon, & \hbox{$ \forall i = 1,2,\cdots, t$;} \\
    E'_{r_i, c_{i+1}} = E_{r_i, c_{i+1}} - \epsilon, E''_{r_i, c_{i+1}} = E_{r_i, c_{i+1}} + \epsilon, & \hbox{$ \forall i =1, 2, \cdots, t-1$.}
  \end{array}
\right.
\nnb
\ee

Under such construction, any row/column sum of $E'$ and $E''$ is equal to that of $E$, except column $c_1$ and row $r_t$.
The column-$c_1$ (resp. row-$r_t$) sum in $E'$ is equal to the column-$c_1$ (resp. row-$r_t$) sum in $E$ plus $\epsilon$
and the column-$c_1$ (resp. row-$r_t$) sum in $E''$ is equal to the column-$c_1$ (resp. row-$r_t$) sum in $E$ minus $\epsilon$.
However, since column $c_1$ (resp. row $r_t$) has only one middle entry $(r_1, c_1)$ (resp. $(r_t, c_t)$) in $E$, the column-$c_1$ (resp. row-$r_t$) sum in $E$ must
be in the interval $(0,1)$. Thus, we can always find a sufficiently small positive real number $\epsilon$ such that
both the column-$c_1$ sum and the row-$r_t$ sum in $E'$ and $E''$ are still in the interval $(0,1)$
and all entries in  $E'$ and $E''$  are in the interval $[0,1/k]$, ensuring $E' \in \mathcal{A}_k$ and $E'' \in \mathcal{A}_k$.
Since again we have $E = (E'+E'')/2$ and $E' \neq E''$, matrix $E$ is not an extreme point of set $\mathcal{A}_k$.
\end{proof}

\section{Remarks}
Our proof is inspired by Watkins and Merris's work \cite{watkins1974convex}. In particular,
the proof of \emph{Case I} is similar to that in \cite{watkins1974convex} for doubly stochastic matrices. However,
the proof of \emph{Case II} is new for doubly substochastic matrices which could have middle lines with only one middle entry.
In fact, the key to proving Theorem~\ref{thm:main-result}
is to divide all possibilities into \emph{Case I} and \emph{Case II} with respect to middle lines.



\section*{References}

\bibliography{ref}

\begin{thebibliography}{1}
\expandafter\ifx\csname url\endcsname\relax
  \def\url#1{\texttt{#1}}\fi
\expandafter\ifx\csname urlprefix\endcsname\relax\def\urlprefix{URL }\fi
\expandafter\ifx\csname href\endcsname\relax
  \def\href#1#2{#2} \def\path#1{#1}\fi

\bibitem{birkhoff}
G.~Birkhoff, {Tres observaciones sobre el algebra lineal}, Universidad Nacional
  de Tucum{\'a}n Revista, Serie A 5 (1946) 147--151.

\bibitem{mirsky1959convex}
L.~Mirsky, On a convex set of matrices, Archiv der Mathematik 10~(1) (1959)
  88--92.
\newblock \href {http://dx.doi.org/10.1007/BF01240767}
  {\path{doi:10.1007/BF01240767}}.

\bibitem{watkins1974convex}
W.~Watkins, R.~Merris, Convex sets of doubly stochastic matrices, Journal of
  Combinatorial Theory, Series A 16~(1) (1974) 129--130.
\newblock \href {http://dx.doi.org/10.1016/0097-3165(74)90079-X}
  {\path{doi:10.1016/0097-3165(74)90079-X}}.

\bibitem{giaquinta2011mathematical}
M.~Giaquinta, G.~Modica, Mathematical Analysis: Foundations and Advanced
  Techniques for Functions of Several Variables, Springer Science \& Business
  Media, 2011.
\newblock \href {http://dx.doi.org/10.1007/978-0-8176-8310-8}
  {\path{doi:10.1007/978-0-8176-8310-8}}.

\end{thebibliography}

\end{document}